\documentclass[12pt]{article}
\usepackage{natbib}
\usepackage{amssymb,latexsym}
\usepackage{graphicx,amsmath,amsfonts,amsthm}
\usepackage{color}
\usepackage{graphicx}
\usepackage{etex}
\usepackage{morefloats}
\usepackage{float}

\input prepictex
\input postpictex
\input pictexwd

\newcommand{\bpic}[4]{\beginpicture
  \setcoordinatesystem units  <1pt,1pt>
  \setplotarea x from #1 to #2, y from #3 to #4}
\newcommand{\epic}{\endpicture}
\newcommand{\hl}[3]{\put{\line(1,0){#1}} [Bl] at #2 #3 }
\newcommand{\vl}[3]{\put{\line(0,1){#1}} [Bl] at #2 #3 }

\newcommand{\bull}[2]{\put{$\bullet$} at #1 #2 }

\newtheorem{proposition}{Proposition}

\newtheorem{theorem}{Theorem}
\newtheorem{example}{Example}
\newtheorem{definition}{Definition}

\newcommand{\cd}{{\mathcal D}}
\newcommand{\ce}{{\mathcal E}}

\newcommand{\cl}{{\mathcal L}}

\title{\bf A composition of Condorcet domains}
\author{Dominic Keehan and Arkadii Slinko}

\date{}

\begin{document}

\maketitle

\begin{abstract}
Inspecting known maximal Condorcet domains on 4 variables classified by \cite{tobias16} we find that  $9$ out of $18$ of them are created using a certain composition of smaller domains. In this paper we describe this composition. We give sufficient conditions for the composition of two Condorcet domain to be a maximal Condorcet domain.  
\end{abstract}
\vspace{5mm}

\section{The composition $\diamond$}

Let $\cd_1$ and $\cd_2$ be two domains of linear orders on sets $\{1,\ldots,n-1\}$ and $\{2,\ldots,n\}$ of alternatives, respectively. Then we define the $nl$-composition ({\em never-last composition}) of these domains as
\[
\cd_1\diamond \cd_2 = \{un \mid u\in \cd_1\}\cup \{v1 \mid v\in \cd_2\}.
\]
Graphically this domain can be displayed as
\begin{equation*}
\label{matrix}
\cd_1\diamond  \cd_2=\left[\begin{array}{ccccc|cccccccccccccccccc}
&& \cd_1(1,\ldots,n-1)&&&&& \cd_2(2,\ldots,n)&& \\
\hline
n&n&\cdots&n&n&1&1&\cdots&1&1
\end{array}\right],
\end{equation*}
where linear orders of the domain are displayed as columns.

We immediately see that some small domains are $nl$-compositions. For example, the unique domain $(1)\star (2)=\{12,21\}$ (notation from  \cite{DanilovK13}) can also be written as $(1)\diamond  (2)$. There are three maximal domains on three alternatives
\begin{equation*}
\label{eq:three_domains}
\cd_{3,1}=\{123, 312, 132, 321\},\ \cd_{3,2}=\{123,231, 132, 321\},\ \cd_{3,3}=\{123, 213, 231, 321\}. 
\end{equation*}
among which one is completely reducible using never-last joins, namely
\[
\cd_{3,3}(1,2,3)=((1)\diamond (2))\diamond ((2)\diamond (3)), 
\]
where $(i)$ is a trivial linear order on the set of a single alternative $i$.
 
As was demonstrated in \cite{slinko2019} Arrow's single-peaked domain is always an $nl$-decomposition of two smaller Arrow's single-peaked domains. But as we will see the use of this construction is not restricted to Arrow's single-peaked domains. \par\medskip

Here we try to answer the following question: Can this composition help us to construct new maximal Condorcet domains and which domains it  is possible to combine to obtain a new maximal Condorcet domain? To answer we need to introduce the following concept.

\begin{definition}
Let $\cd$ be a Condorcet domain. We say that in a linear order $w= \ldots bc\ldots a \ldots\in \cd$ alternative $a$ is a {\em right  obstruction}  to the swap $bc\to cb$, if $cba$ cannot be potentially in the restriction $\cd|_{\{a,b,c\}}$ of $\cd$ onto $\{a,b,c\}$, that is the domain $\cd|_{\{a,b,c\}}\cup \{cba\}$ is not Condorcet. 
\end{definition}

\begin{proposition}
Let $\cd\subseteq \cl(A)$ be a Condorcet domain and $a,b,c\in A$.  An alternative $a$ is a right  obstruction to the swap $bc\to cb$ if and only if  $\cd|_{\{a,b,c\}}$ satisfies either $cN_{\{a,b,c\}}1$ or $bN_{\{a,b,c\}}2$ (or both) and no other never condition.
\end{proposition}

\begin{proof}
If we allow $bca$ but not $cba$ in $\cd|_{\{a,b,c\}}$ overall we might have four more linear orders listed as columns of the following matrix
\[
\left[\begin{array}{ccccc}
b&b&a&a&c\\
c&a&b&c&a\\
a&c&c&b&b
\end{array}\right]
\]
Only the third and the fifth column being removed give us a Condorcet domain. These domains satisfy $cN_{\{a,b,c\}}1$ or $bN_{\{a,b,c\}}2$, respectively. Either of the two preclude order $cba$.
\end{proof}

\begin{theorem}
\label{thm:maxmax}
If $\cd_1$ and $\cd_2$ are Condorcet domains on $[n-1]$ and $[n]\setminus \{1\}$, respectively, such that  $\ce=(\cd_1)_{-(n-1)} \cup (\cd_2)_{-1}$ is also Condorcet. 
Suppose $n$ is not a right obstruction in $\cd_2$ and $1$ is not a right obstruction in $\cd_1$. Then $\cd=\cd_1\diamond  \cd_2$ is also a Condorcet domain. Moreover, if $\cd_1$ and $\cd_2$ are maximal and ample, then $\cd$  is also maximal and ample. If $\cd_1$ and $\cd_2$ are copious, then $\cd$ is copious as well. 
\end{theorem}

\begin{proof}
Let us prove $\cd$ is Condorcet. Let $i,j,k\in [n]\setminus \{1,n\}$. Then $\cd|_{\{i,j,k\}}$ is Condorcet as $\ce$ is. If $i,j\in [n]\setminus \{1,n\}$, consider the triple $\{i,j,n\}$. Since $n$ is not a right obstruction to both swaps $ij\to ji$ and $ji\to ij$ in $\cd_2$ we have $ijn$ and $jin$ compatible with $\cd_2|_{\{i,j,n\}}$, that is $\cd_2|_{\{i,j,n\}}\cup \{ijn,jin\}$ is Condorcet. However, only these two orders can be additionally added from $\cd_1$ with $n$ at the bottom, that is $\cd|_{\{i,j,n\}} \subseteq \cd_2|_{\{i,j,n\}}\cup \{ijn,jin\}$ is Condorcet.
If $i\in [n]\setminus \{1,n\}$, then the triple $\{1,i,n\}$ is Condorcet since it satisfies $iN_{\{1,i,n\}}3$.

If $\cd_1$ and $\cd_2$ are maximal, we cannot add to $\cd$ an order ending with 1 or $n$. Let $1<j<n$. Then due to ampleness of $\cd_1$  orders $1jn$, $j1n$ are contained in $\cd$ and due to ampleness of $\cd_2$ orders $jn1$, $nj1$ are contained in $\cd$, hence $\cd$ is copious and satisfies $jN_{\{1,j,n\}}3$. Hence adding to $\cd$ an order with $j$ as its last alternative, would violate $iN_{\{1,i,n\}}3$. 

Suppose $\cd_1$ and $\cd_2$ are copious. If $\{i,j,k\}\subset [n-1]$, then $|\cd|_{\{i,j,k\}}|=4$ as $\cd_1$ is copious. Similarly, $|\cd|_{\{i,j,k\}}|=4$ if 
$\{i,j,k\}\subset [n]\setminus \{1\}$. If $\{i,j,k\}\supset \{1,n\}$, then $|\cd|_{\{i,j,k\}}|=4$ as we have seen already. This proves the theorem.
\end{proof}

This theorem gives us only  a sufficient set of conditions. As we will see the $nl$-composition can be maximal without $\cd_1$ or $\cd_2$ being maximal.  This, in particular, means that $nl$-joins construction does not always produce a Condorcet domain.

\section{Domains on four alternatives obtained by never-last joins}

We illustrate the result using the classification by \cite{tobias16}, this paper will be referred as TD. He determined all 18 maximal Condorcet domains on four alternatives up to an isomorphism and flip-isomorphism. Using the order given by Table 2 in their paper we denote those maximal domains as $\cd_{4,1},\ldots, \cd_{4,18}$. We will show that nine out of 18 domains in that list are obtained by using $nl$-composition. These domains are: $\cd_{4,4}$ (The Single Peaked); $\cd_{4,5}$ (The Crab); $\cd_{4,6}$ (The Sun); $\cd_{4,7}$ (The Half-Crab-Half-Sun); $\cd_{4,11}$ (Boring I), $\cd_{4,16}$ (Boring VI), $\cd_{4,17}$ (Boring VII); $\cd_{4,2}$ (The Snake); $\cd_{4,3}$ (The Broken Snake).\footnote{The names of the domains are far from being optimal.}

\begin{example}[The Single-peaked in TD]
\label{ex:4-SP}
Let us consider a unique maximal single-peaked domain $\mathcal{SP}(\triangleleft,A)$ on four alternatives with $A=\{1,2,3,4\}$ relative to spectrum $1\triangleleft 2\triangleleft 3\triangleleft 4$:
\begin{equation*}
\label{eq:single-peaked}
\left[\begin{array}{cccc|cccc}  1&2&2&3&2&3&3&4\\2&1&3&2&3&2&4&3\\3&3&1&1&4&4&2&2\\\hline 4&4&4&4&1&1&1&1  
\end{array}\right]
\end{equation*}
whose graph is presented on Figure~\ref{fig:SP-4alts}.  

\begin{figure}[H]
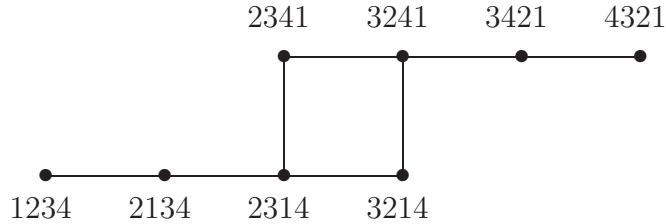

\centering
$$\bpic{-20}{250}{-20}{50}
\bull{0}{0}
\put{$1234$} at -2 -12
\hl{45}{0}{0}
\bull{45}{0}
\put{$2134$} at 43 -12
\put{$2314$} at 88 -12
\put{$3214$} at 133 -12
\bull{135}{45}
\bull{135}{0}
\bull{180}{45}
\bull{225}{45}
\vl{45}{135}{0}
\hl{90}{45}{0}
\bull{90}{0}
\vl{45}{90}{0}
\bull{90}{45}
\hl{135}{90}{45}
\put{$2341$} at 88 60
\put{$3241$} at 133 60
\put{$3421$} at 178 60
\put{$4321$} at 223 60
\epic$$
\caption{Graph of the single-peaked domain $\mathcal{SP}(\triangleleft,A)$ on four alternatives}
\label{fig:SP-4alts}
\end{figure}

We have $\mathcal{SP}(\triangleleft,A)=\cd(\mathcal{N})$, where 
\[
\mathcal{N}=\{bN_{\{a,b,c\}}3,\  bN_{\{a,b,d\}}3,\  cN_{\{a,c,d\}}3,\ cN_{\{b,c,d\}}3 \}.
\]
This domain will also be denoted as $\cd_{4,1}$. We have
\[
\cd_{4,4}=\cd_{3,3}(1,2,3)\diamond  \cd_{3,3}(2,3,4).
\]
\end{example}

\begin{example}[The Crab in TD]
\label{ex:4-strange}
Let us consider an  Arrow's single-peaked maximal Condorcet domain $\cd_{4,5}$ on four alternatives:
\begin{equation*}
\label{eq:single-peaked}
\left[\begin{array}{cccc|cccc}
1&2&2&3&2&3&2&4\\
2&1&3&2&3&2&4&2\\
3&3&1&1&4&4&3&3\\
\hline
4&4&4&4&1&1&1&1
\end{array}\right]
\end{equation*}
\[
\cd_{4,5}= \cd_{3,3}(1,2,3)\diamond  \cd_{3,3}(3,2,4).
\]
whose graph is presented on Figure~\ref{fig:4-strange1}. We have $\cd_{4,2}=\cd(\mathcal{N})$, where 
\[
\mathcal{N}=\{2N_{\{1,2,3\}}3,\  2N_{\{1,2,4\}}3,\  3N_{\{1,3,4\}}3,\ 2N_{\{2,3,4\}}3 \}
\]
and $\cd_{4,2}$ is copious.  This domain is not single-peaked (for example, because it does not have two completely reversed orders).

\begin{figure}[H]
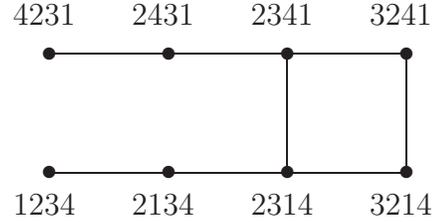

\centering
$$\bpic{-20}{180}{-20}{50}
\bull{0}{0}
\put{$1234$} at -2 -12
\hl{45}{0}{0}
\bull{45}{0}
\put{$2134$} at 43 -12
\put{$2314$} at 88 -12
\put{$3214$} at 133 -12
\bull{135}{45}
\bull{135}{0}
\bull{0}{45}
\bull{45}{45}
\vl{45}{135}{0}
\hl{90}{45}{0}
\bull{90}{0}
\vl{45}{90}{0}
\bull{90}{45}
\hl{45}{90}{45}
\hl{90}{0}{45}
\put{$2431$} at 43 60
\put{$3241$} at 133 60
\put{$2341$} at 88 60
\put{$4231$} at -2 60
\epic$$
\caption{\label{fig:4-strange1} Graph of an Arrow's single-peaked domain $\cd_{4,5}$.}
\end{figure}
\end{example}

\begin{example}[The Sun in TD]
Let us consider the following maximal Condorcet domain on four alternatives:
\begin{equation*}
\label{eq:sun}
\left[\begin{array}{cccc|cccc}
3&3&2&2&3&3&2&2\\
1&2&3&1&4&2&3&4\\
2&1&1&3&2&4&4&3\\
\hline
4&4&4&4&1&1&1&1
\end{array}\right]
\end{equation*}
It is defined by the following complete set of never-conditions:
\[
1N_{\{1,2,3\}}1,\quad   2N_{\{1,2,4\}}3, \quad 3N_{\{1,3,4\}}3, \quad 4N_{\{2,3,4\}}1.
\]
It is copious but does not have maximal width. The median graph of this domain is here:
\begin{figure}[H]
\centering
$$\bpic{-20}{180}{-20}{50}
\bull{0}{0}
\put{$2134$} at -2 -12
\put{$2314$} at 43 -12
\put{$2341$} at 88 -12
\put{$2431$} at 133 -12
\bull{45}{0}
\bull{135}{45}
\bull{135}{0}
\bull{0}{45}
\bull{45}{45}
\bull{90}{0}
\bull{90}{45}
\vl{45}{45}{0}
\vl{45}{90}{0}
\hl{45}{0}{0}
\hl{90}{45}{0}
\hl{45}{90}{45}
\hl{90}{0}{45}
\put{$3124$} at -2 60
\put{$3214$} at 43 60
\put{$3421$} at 133 60
\put{$3241$} at 88 60
\epic$$
\caption{\label{fig:4-ladder} Graph of $\cd_{4,6}$.}
\end{figure} 
In terms of $nl$-joins
\[
\cd_{4,6}= \cd_{3,1}(2,1,3)\diamond  \cd_{3,1}(2,4,3).
\]
\end{example}

\begin{example}[Half Crab half Sun domain in TD] 
Let us consider the following maximal Condorcet domain on four alternatives:
\begin{equation*}
\label{eq:sun}
\left[\begin{array}{cccc|cccc}
3&3&2&2&4&3&3&2\\
1&2&3&1&3&4&2&3\\
2&1&1&3&2&2&4&4\\
\hline
4&4&4&4&1&1&1&1
\end{array}\right]
\end{equation*}

It is defined by the following complete set of never-conditions:
\[
3N_{\{1,2,3\}}1,\quad   1N_{\{1,2,4\}}3, \quad 1N_{\{1,3,4\}}3, \quad 2N_{\{2,3,4\}}3.
\]
It is copious but does not have maximal width.
\begin{figure}[H]
\centering
$$\bpic{-20}{180}{-20}{50}
\bull{0}{0}
\put{$2134$} at -2 -12
\put{$2314$} at 43 -12
\put{$2341$} at 88 -12
\bull{45}{0}
\bull{135}{45}
\bull{0}{45}
\bull{45}{45}
\bull{90}{0}
\bull{90}{45}
\bull{180}{45}
\vl{45}{45}{0}
\vl{45}{90}{0}
\hl{45}{0}{0}
\hl{45}{45}{0}
\hl{90}{90}{45}
\hl{90}{0}{45}
\put{$3124$} at -2 60
\put{$3214$} at 43 60
\put{$3241$} at 88 60
\put{$3421$} at 133 60
\put{$4321$} at 178 60
\epic$$
\caption{\label{fig:b-ladder} Graph of $\cd_{4,7}$.}
\end{figure} 
In terms of $nl$-joins
\[
\cd_{4,7}= \cd_{3,1}(2,1,3)\diamond  \cd_{3,1}(2,3,4).
\]
\end{example}

\begin{example}[Boring I in TD]
Let us consider the following maximal Condorcet domain for $m=4$ alternatives:
\begin{equation*}
\label{eq:single-peaked}
\left[\begin{array}{cccc|cccc}  1&2&2&3&2&3&4&4\\2&1&3&2&3&2&2&3\\3&3&1&1&4&4&3&2\\\hline 4&4&4&4&1&1&1&1  
\end{array}\right]
\end{equation*}
which satisfies never-conditions
\[
2N_{\{1,2,3\}}3, \quad 2N_{\{1,2,4\}}3, \quad 3N_{\{1,3,4\}}3, \quad 4N_{\{2,3,4\}}2.
\]
In terms of $nl$-joins
\[
\cd_{4,11}= \cd_{3,3}(1,2,3)\diamond  \cd_{3,2}(4,2,3).
\]
Figure~\ref{fig:nl-composition-b1} shows the graph of this domain.
\begin{figure}[H]
\centering
$$\bpic{-20}{180}{-10}{95}
\bull{0}{0}
\put{$1234$} at -2 -12
\hl{45}{0}{0}
\bull{45}{0}
\put{$2134$} at 43 -12
\put{$2314$} at 88 -12
\put{$3214$} at 133 -12
\bull{135}{45}
\bull{135}{0}
\bull{90}{90}
\bull{135}{90}
\hl{90}{45}{0}
\bull{90}{0}
\plot 90 45 90 90 /
\plot 135 45 135 90 /
\bull{90}{45}
\hl{45}{90}{45}
\hl{45}{90}{90}
\put{$2341$} at 73 48
\put{$3241$} at 153 48
\put{$4231$} at 73 92
\put{$4321$} at 153 92
\setlinear 
\plot 135 0 135 45 /
\plot 90 0 90 45 /
\epic$$
\caption{\label{fig:nl-composition-b1} Graph of $\cd_{4,11}$.} 
\end{figure}
\end{example}


\begin{example}[Boring VI in TD]
Let us consider the following maximal Condorcet domain for $m=4$ alternatives:
\begin{equation*}
\label{eq:boring6}
\left[\begin{array}{cccc|cccc}  1&1&2&3&3&2&4&4\\2&3&3&2&2&3&3&2\\3&2&1&1&4&4&2&3\\\hline 4&4&4&4&1&1&1&1  
\end{array}\right]
\end{equation*}
which satisfies never-conditions
\[
1N_{\{1,2,3\}}2, \quad 2N_{\{1,2,4\}}3, \quad 3N_{\{1,3,4\}}3, \quad 4N_{\{2,3,4\}}2.
\]
Figure~\ref{fig:nl-composition-b6} shows the median graph of the domain. 
In terms of $nl$-joins
\[
\cd_{4,16}= \cd_{3,3}(1,2,3)\diamond  \cd_{3,2}(4,2,3).
\]
\begin{figure}[H]
\centering
$$\bpic{-50}{180}{-10}{60}
\bull{0}{0}
\put{$1234$} at -2 -12
\put{$1324$} at -2   57
\put{$3214$} at  43   57
\hl{45}{0}{0}
\bull{45}{0}
\put{$2314$} at 43 -12
\put{$2341$} at 88 -12
\put{$4231$} at 133 -12
\bull{135}{45}
\bull{135}{0}
\hl{90}{45}{0}
\bull{90}{0}
\bull{90}{45}
\hl{45}{90}{45}
\put{$3241$} at 88 57
\put{$4321$} at 133 57 
\setlinear 
\plot 135 0 135 45 /
\plot 90 0 90 45 /
\plot 0 0 0 45 90 45 /
\plot 45 0 45 45 /
\bull{0}{45}
\bull{45}{45}
\epic$$
\caption{\label{fig:nl-composition-b6} Graph of $\cd_{4,16}$.} 
\end{figure}
\end{example}

\newpage

\begin{example}[Boring VII in TD]
Let us consider the following maximal Condorcet domain for $m=4$ alternatives:
\begin{equation*}
\label{eq:single-peaked}
\left[\begin{array}{cccc|cccc}  2&3&2&3&2&3&4&4\\1&1&3&2&3&2&2&3\\3&2&1&1&4&4&3&2\\\hline 4&4&4&4&1&1&1&1  
\end{array}\right]
\end{equation*}
which satisfies never-conditions
\[
1N_{\{1,2,3\}}1, \quad 2N_{\{1,2,4\}}3, \quad 3N_{\{1,3,4\}}3, \quad 4N_{\{2,3,4\}}2.
\]
Figure~\ref{fig:boring7} shows the median graph of the domain. 
In terms of $nl$-joins
\[
\cd_{4,17}= \cd_{3,1}(2,1,3)\diamond  \cd_{3,2}(4,2,3).
\]
\end{example}

Figure~\ref{fig:boring7} shows the median graph of the domain. 
In terms of $nl$-joins
\begin{figure}[h]
\centering
$$\bpic{40}{180}{-10}{95}
\bull{180}{0}
\hl{45}{135}{0}
\bull{45}{0}
\put{$2134$} at 43 -12
\put{$2314$} at 88 -12
\put{$3214$} at 133 -12
\put{$3124$} at 180 -12
\bull{135}{45}
\bull{135}{0}
\bull{90}{90}
\bull{135}{90}
\hl{90}{45}{0}
\bull{90}{0}
\plot 90 45 90 90 /
\plot 135 45 135 90 /
\bull{90}{45}
\hl{45}{90}{45}
\hl{45}{90}{90}
\put{$2341$} at 73 48
\put{$3241$} at 153 48
\put{$4231$} at 73 92
\put{$4321$} at 153 92
\setlinear 
\plot 135 0 135 45 /
\plot 90 0 90 45 /
\epic$$
\caption{\label{fig:boring7} Graph of $\cd_{4,17}$.} 
\end{figure}

However, the converse of Theorem~\ref{thm:maxmax} is not correct as the following example shows. 

\begin{example}[The snake in TD]
Let us consider a single-crossing maximal domain $\cd_{4,2}$ whose orders are represented as columns of the following matrix
\begin{equation*}
\label{eq:single-crossing7}
\left[\begin{array}{ccc|cccc}  1&2&2&2&2&4&4\\2&1&3&3&4&2&3\\3&3&1&4&3&3&2\\\hline 4&4&4&1&1&1&1  \end{array}\right]
\end{equation*}
We see that 
\[
\cd_{4,2}=\ce\diamond  \cd_{3,1}(2,3,4),
\] 
where 
$
\ce=\{123, 213, 231\}\subseteq \cd_{3,3}(1,2,3)\cap \cd_{3,1}(2,3,1).
$
and $\cd_1$ is not maximal. This happens due to $4$ being  an obstruction to swap $23\to 32$ in $\cd_{3,1}(2,3,4)$ hence we cannot add order $321$ or $132$ to $\cd_1$ to make it maximal. 

The graph of the single-crossing domain is a line graph shown on a Figure~\ref{fig:SC-4alts}.
 \begin{figure}[H]
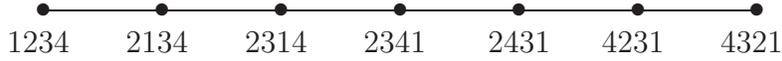

\centering
$$\bpic{-20}{280}{-20}{5}
\bull{0}{0}
\bull{45}{0}
\bull{90}{0}
\bull{135}{0}
\bull{180}{0}
\bull{225}{0}
\bull{270}{0}
\hl{270}{0}{0}
\put{$1234$} at -2 -12
\put{$2134$} at 43 -12
\put{$ 2314$} at 88 -12
\put{$2341$} at 133 -12
\put{$2431$} at 180 -12
\put{$4231$} at 223 -12
\put{$4321$} at 268 -12
\epic$$
\caption{\label{fig:SC-4alts} Graph of single-crossing domain $SC_4$ on four alternatives}
\end{figure}
\end{example}

One other such example.

\begin{example}[The broken snake in TD]
Let us consider the maximal domain $\cd_{4,3}$ whose orders are represented as columns of the following matrix
\begin{equation*}
\label{eq:single-crossing7}
\left[\begin{array}{ccc|cccc}  1&3&3&3&3&2&4\\3&1&2&2&4&4&2\\2&2&1&4&2&3&3\\\hline 4&4&4&1&1&1&1  \end{array}\right]
\end{equation*}
We see that 
\[
\cd_{4,3}=\ce\diamond  \cd_{3,2}(3,2,4),
\]
 where 
\[
\ce=\{132, 312, 321\}\subseteq \cd_2=\{123,132,231,321\}.
\]
and $\ce$ is not maximal. This happens due to $4$ being  an obstruction to swap $32\to 23$ in $\cd_{3,2}(3,2,4)$ hence we cannot add order $123$ or $231$ to $\ce$ to make it maximal. 

\begin{figure}[h]
\centering
$$\bpic{-20}{180}{-10}{95}
\bull{45}{0}
\put{$3214$} at 43 -12
\put{$3124$} at 28 45
\put{$1324$} at 28 90
\put{$3241$} at 88 -12
\put{$3421$} at 133 -12
\bull{135}{45}
\bull{135}{0}
\bull{45}{45}
\bull{45}{90}
\hl{90}{45}{0}
\vl{90}{45}{0}
\bull{90}{0}
\bull{90}{45}
\hl{45}{90}{45}
\put{$2431$} at 90 57
\put{$4231$} at 135 57
\setlinear 
\plot 135 0 135 45 /
\plot 90 0 90 45 /
\epic$$
\caption{\label{fig:nl-composition-with_a_twist} Graph of $\cd_{4,3}$.} 
\end{figure}
\end{example}

\bibliographystyle{plainnat}
\bibliography{cps}
\end{document}